\documentclass{amsart}

\usepackage[top=30truemm,bottom=30truemm,left=30truemm,right=30truemm]{geometry}

\usepackage[dvipdfmx]{hyperref}
\usepackage{amscd}
\usepackage{amsmath, amssymb}
\usepackage{mathrsfs}
\usepackage[all]{xy}
\usepackage{amsthm}
\usepackage{color}
\usepackage{paralist}
\usepackage{enumerate}
\makeindex
\usepackage{latexsym}

\hypersetup{setpagesize=true, bookmarksnumbered=false, bookmarksopen=true, colorlinks=true,linkcolor=red, citecolor=green,}

\newtheorem{thm}{Theorem}[section]

\newtheorem{cor}[thm]{Corollary}
\newtheorem{lem}[thm]{Lemma}

\newtheorem{cl}[thm]{Claim}
\newtheorem{lemm}[thm]{Lemma}

\theoremstyle{definition}
\newtheorem{dfn}[thm]{Definition}
\newtheorem*{dfn*}{Definition}
\newtheorem{nota}[thm]{Notation}

\theoremstyle{remark}
\newtheorem{rmk}[thm]{Remark}

\newcommand{\tf}{\mathscr{M}^{\op{tf}}(v)}

\newcommand{\sss}{\mathscr{M}^{\op{ss}}(v)}
\newcommand{\msss}{\mathscr{M}^{\mu\text{ss}}(v)}
\newcommand{\hn}{\mathscr{M}_{(v_1,v_2)}^{\op{HN}}(v)}

\newcommand{\la}{\langle}
\newcommand{\ra}{\rangle}
\newcommand{\op}{\operatorname}

\begin{document}

\title[Moduli stacks of torsion free sheaves and Brill-Noether theory]
{Classifying the irreducible components of moduli stacks of torsion free sheaves on K3 surfaces and an application to Brill-Noether theory}
\author{Yuki Mizuno}

\email{m7d5932a72xxgxo@fuji.waseda.jp}
\address{Department~of~Mathematics, School~of~Science~and~Engineering, Waseda~University, Ohkubo~3-4-1, Shinjuku, Tokyo~169-8555, Japan}

\begin{abstract}
In this article, we classify the irreducible components of moduli stacks of torsion free sheaves of rank 2 on K3 surfaces of Picard number 1. For ruled surfaces, the components of moduli stacks of torsion free sheaves were classified by Walter (\cite{walter1995components}). Moreover, by virtue of our result, we classify the irreducible components of Brill-Noether loci of Hilbert schemes of points on K3 surfaces.
\end{abstract}

\keywords{Moduli spaces of sheaves , Algebraic stacks , K3 surfaces , Brill-Noether theory}
\subjclass[2010]{14D20, 14D23, 14J28}

\maketitle

\section{Introduction}
 Moduli spaces of sheaves is one of the most central areas of algebraic geometry. By considering them, many interesting objects have been found. On K3 surfaces, moduli spaces of sheaves can have symplectic structures, which was first observed by Mukai (\cite{mukai1984symplectic}). On the other hand, as is well-known, we can construct such moduli spaces by restricting objects to coherent sheaves satisfying stability. However, the moduli spaces do not parametrize unstable sheaves.
 In this point, stack is important and useful tool to construct moduli spaces which is difficult to construct in the framework of scheme. 
 
 Our original motivation of the present paper is studying symplecticity of moduli spaces of sheaves on K3 surfaces. 
 Moreover, in \cite{mukai1984moduli} and \cite{yoshioka1999irreducibility} and others, it was shown that non-emptiness, irreducibility and other properties of moduli schemes depend essentially on Mukai vector. In \cite{kurihara2008holomorphic} and \cite{yoshioka2003twisted}, properties of the moduli stacks of semistable sheaves on K3 surfaces are studied. Although we can study moduli spaces of unstable sheaves on K3 surfaces by using stack theory, detailed observations are less than studies of moduli schemes.
  
  Various types of stratifications of stacks are studied by G\'omez, Sols and Zamora \cite{gomez2015git} and Hoskins \cite{hoskins2018stratifications} and others. However, it seems that irreducible decomposition of moduli stacks of sheaves is not treated in these papers. In the present article, we first classify the irreducible components of moduli stacks of torsion free sheaves of rank 2 on K3 surfaces of Picard number $\rho = 1$. 
  Classifying the irreducible components of moduli stacks of torsion free sheaves on ruled surfaces 
  is discussed in \cite{walter1995components}.
  However, we need new ideas to solve our problem because K3 surfaces have trivial canonical sheaves and may not be fibered surfaces. Important results and methods in this paper are studies of moduli stacks of semistable sheaves and filtered sheaves by Yoshioka (\cite{kimura2011}, \cite{kurihara2008holomorphic}, \cite{yoshioka2003twisted}, \cite{yoshioka2009fourier}), the classical theory by Shatz (\cite{shatz1977decomposition}) and generalized Shatz's theory by Nitsure (\cite{nitsure2011schematic}). By using these theories, we obtain our first result. More precisely, we first take stratification of moduli stacks of torsion free sheaves by moduli stacks of semistable sheaves and ones of Harder-Narasimhan filtrations. After that, we analyze the strata and describe the irreducible components by using the above theory of Yoshioka.

  If $\tf$ and $\sss$ denote respectively the moduli stacks of torsion free sheaves and semistable sheaves with Mukai vector $v$ (in detail, see Definition \ref{def:mv}), our first result is the following. 
  
 \begin{thm}\label{thm:mainthm1}
Let
 $X$ be a \text{K3 surface} of $\rho(X)=1$ over $\mathbb{C}$, let $v_0$ be a primitive Mukai vector and, let $v =([v]_0, [v]_1, [v]_2):=mv_0 $ $(m \in \mathbb{Z})$. 
We assume $[v]_0=2$. 
Then, we have the irreducible decomposition of $\tf$ as follows.
	\begin{equation*}
		 \tf = \begin{cases}
\overline{\sss} \cup \bigcup_{(v_1, v_2) \leq 1 } \overline{\mathscr{M}_{(v_1,v_2)}^{\text{HN}}(v)} & \text{if }\la v_0, v_0 \ra \geq -2\\
\bigcup \overline{\mathscr{M}_{(v_1,v_2)}^{\text{HN}}(v)}& otherwise
\end{cases}
	\end{equation*}	
 , where 
 the stack $\hn$ is defined as
	\[\hn :=
	\left\{ E \in  \tf \mathrel{} \middle| \mathrel{} \begin{gathered} ^\exists(0 \subset E_1 \subset E) :\text{Harder-Narasimhan filtration}\\
	\text{such that } v(E_1)=v_1,v(E/E_1)=v_2 \end{gathered}
	\right\}.
	\]
\end{thm}
We call $\hn$ the moduli stack of Harder-Narasimhan filtrations with type $(v_1, v_2)$. (in detail, see Definition \ref{dfn:mshn})

\begin{rmk}
	Note that $\sss \neq \emptyset$ if and only if $\la v_0, v_0 \ra \geq -2$ (\cite[Corollary 0.3]{yoshioka1999irreducibility}).
	And, we can compute the dimensions of $\tf$ at each point by using Theorem \ref{thm:mainthm1} and Lemma \ref{lemm:dimension formula}.
\end{rmk}

The second purpose of this paper is classifying the irreducible components of Brill-Noether loci of Hilbert schemes of points on K3 surfaces by using the first result. Originally, in \cite{walter1995components}, components of Brill-Noether loci of Hilbert schemes of points on ruled surfaces were classified. In \cite{walter1995components}, Castelnuovo-Mumford regularity and the Bertini's theorem were mainly used. However, we need more detailed analysis to achieve the application for K3 surfaces.
Namely, we focus on the method of the proof of the Bertini theorem (\cite{badescu2012projective}) and more recent results about K3 surfaces (\cite{kurihara2008holomorphic}, \cite{yoshioka1999irreducibility}). 
Our second result is the following.
\begin{thm}\label{thm:mainthm2}
	Let $X$ be a K3 surface of $\rho(X) = 1$ over $\mathbb{C}$, let $v:=(2, nH, \frac{n^2}{2}H^2-N+2)$$ = mv_0 $ $(v_0: \text{primitive}$ $ \text{Mukai vector}$, $m \in  \mathbb{Z})$
	and let $nH$ be an effective divisor on $X$ $(n \in \mathbb{Z}_{\geq 0}, H : \text{the }\text{generator }\text{of }$$\op{Pic(X)})$. We assume $N \leq h^0(\mathscr{O}(nH))$. 
Then, we classify the irreducible components of
	\[ W^0_N(nH) = \{ [Z] \in \op{Hilb}^N(X) \mid h^1(\mathscr{I}_Z(nH)) \geq 1 \} \]
into one of the following.

$(\alpha): $ for all $(v_1, v_2)$, if 
 $\la v_1, v_2 \ra \leq 1$, $[v_1]_1, [v_2]_1 \neq 0 : \text{effective}$ and $-1 < [v_2]_2$, there exists a unique irreducible component of $ W^0_N(nH)$ such that, for a general member $Z$, the torsion free sheaf $E$ fitting into the extension 
\[ 0 \rightarrow \mathscr{O}_X \rightarrow E \rightarrow \mathscr{I}_Z(nH) \rightarrow 0\]
is contained in $\mathscr{M}_{(v_1, v_2)}^{\text{HN}}(v)$.

$(\beta):$ if $\la v_0, v_0 \ra \geq -2$ except for the case `` $H^2= 2$ and $v = (2, 3H, 5)$ '' , 
there exists a unique irreducible component of $ W^0_N(nH)$ such that for a general member $Z$, the torsion free sheaf $E$ fitting into the extension
\[ 0 \rightarrow \mathscr{O}_X \rightarrow E \rightarrow \mathscr{I}_Z(nH) \rightarrow 0\]
is contained in $\mathscr{M}^{\text{ss}}(v)$.

 \end{thm}
 
 \begin{rmk}
  If $N > h^0(\mathscr{O}(nH))$, then $W^0_N(nH) = \op{Hilb}^N(X)$.
 And by using Theorem \ref{thm:mainthm2}, we see not only whether   $W^0_N(nH)$ is empty or not but also the dimensions and the number of the irreducible components of $W^0_N(nH)$.
 \end{rmk}
\begin{rmk}
 About what happens in the exceptional case `` $H^2 = 2$ and $v = (2 , 3H , 5 )$ '' in Theorem \ref{thm:mainthm2}, see Claim \ref{163847_3Dec21} and a few paragraphs after that.
\end{rmk}

\section{Preliminaries}	

In this paper, the word \textit{a surface} means a two-dimensional algebraic variety over $\mathbb{C}$. The word \textit{an algebraic stack} means an Artin stack over $\mathbb{C}$. In addition, the word \textit{open (resp. closed, resp. locally closed) substack} means a strictly substack whose inclusion map is an open (resp. closed, resp. locally closed) immersion (in detail, see \cite{laumon2000moret} or \cite{stacks-project}). 

	\subsection{Mukai vectors}
	\begin{dfn}[Mukai vectors \cite{huybrechts2010geometry}]\label{def:mv}
		Let $X$ be a K3 surface and let $E$ be a coherent sheaf on $X$.Then the Mukai vector $v(E)$ of $E$ is $ ($rank($E$), $c_1(E), \frac{c_1(E)^2}{2}-c_2(E)+$rank$(E)) $ $\in \mathbb{Z}$ $\oplus$ Pic($X$) $\oplus$ $\mathbb{Z}$.
	\end{dfn}
		\begin{dfn}[Mukai paring \cite{huybrechts2010geometry}]
		Let $X$ be a K3 surface and let $ v:=([v]_0,[v]_1,[v]$ $_2)$, $v':=([v']_0,[v']_1,[v']_2)$ $\in $ $\mathbb{Z} \oplus \op{Pic}(X) \oplus \mathbb{Z}$.
		Then, we define the Mukai pairing of $v$ and $v'$ to be  $\la v,v'\ra$$:=-[v]_0[v']_2+[v]_1[v']_1-[v]_2[v']_0$ $\in$$\mathbb{Z}$.	
		\end{dfn}

\begin{dfn}[(\cite{huybrechts2010geometry})]For any $v \in \mathbb{Z} \oplus \op{Pic}(X) \oplus \mathbb{Z}$, $v$ is primitive if “$v' \in \mathbb{Z} \oplus \op{Pic}(X) \oplus \mathbb{Z}$, $m \in \mathbb{Z}$, $v = mv'$
		 $\Rightarrow$ $m =1 \text{ or } -1$”
\end{dfn}

\subsection{Moduli stacks}
\begin{dfn}[Moduli stacks of torsion free sheaves]
\label{mstf}
Let $X$ be a K3 surface over $\mathbb{C}$, and let
$v $ $\in \mathbb{Z} \oplus \op{NS}(X) \oplus \mathbb{Z}$.
we define the moduli stack $\tf$ of torsion free sheaves with Mukai vector $v$ on $X$ to be the following category 
\begin{enumerate}

\item Objects$:(S$, $E)$, where
 $S :$ scheme over $\mathbb{C}$, $E:$ quasi-coherent locally of finite presentation sheaves over $X\times_{\mathbb{C}}S(=:Z)$ and flat over $S$, and $E_t: $ torsion-free sheaf over $Z_t = X_{k(t)}$ such that $v(E) = v$, $(\forall t \in S)$;

\item Morphisms $:$ morphisms from $(S$, $E)$ to $(S'$, $E')$ are the pairings $(\varphi:S \rightarrow S',\alpha:{\varphi}^*E \rightarrow E')$ such that $\alpha$ is an isomorphism. 
\end{enumerate}
\end{dfn}
\begin{rmk}
$\mathscr{M}^{tf}(v)$ is an algebraic stack. And, we can define moduli stacks  $\mathscr{M}(v)$ of coherent sheaves with Mukai vector $v$ on $X$ in the same way.
\end{rmk}

\begin{dfn}[Points of algebraic stacks \cite{laumon2000moret}, \cite{stacks-project}]Let $\mathscr{X}$ be an algebraic stack. Then, 
\[ 
|\mathscr{X}|
 := \coprod_{K/\mathbb{C}: \text{extension of fields} }\mathscr{X}(\op{
 Spec}(K))/ \sim,
	 \]
where if let $E \in \mathscr{X}(\op{Spec}(K))$, let $
	  E' \in \mathscr{X}(\op{Spec}(K'))$ and let $
	  K,K'$ be extensions of $\mathbb{C}$, we write $E \sim E' $ if there exists a extension $K''$ of  $K,K'$ such that $E \mid_{X_{\op{Spec}(K'')}} \simeq E' \mid_{X_{\op{Spec}(K'')}} $.
\end{dfn}

\begin{dfn}[Topological spaces of algebraic stacks \cite{laumon2000moret}, \cite{stacks-project}]
Let $\mathscr{X}$ be an algebraic stack. Then the set $\{U \subseteq  |\mathscr{X}| \mid \exists \mathscr{U}$: open substack of $\mathscr{X}$ such that $|\mathscr{U}| = U \}$ satisfies the axiom of open sets of $\mathscr{X}$. We think of $|\mathscr{X}|$ as a topological space by applying the definition.
\end{dfn}

\begin{dfn}[Relative dimensions \cite{laumon2000moret},\cite{stacks-project}]
	Let $P:U \rightarrow \mathscr{X} $ be a morphism from a scheme, and we assume $u \in U $ maps to $x \in |\mathscr{X}|$.
	Then, we define $\op{dim}_u(P)$ as follows.
	In the commutative diagram
	 \[
    \xymatrix{
	     U\times_{\mathscr{X}}\op{Spec}(k) \ar[r] \ar[d] & \op{Spec}(k) \ar[d]^x \\
     U \ar[r]^P &  \mathscr{X}\ar@{}[lu]|{\Box},}
	\] 
	\[
	\op{dim}_u(P) := \op{dim}_x( U\times_{\mathscr{X}}\op{Spec}(k)).
	\]
	\end{dfn}

\begin{dfn}[Dimensions of algebraic stacks at points \cite{laumon2000moret}, \cite{stacks-project}]
Let $\mathscr{X}$ be an algebraic stack, let $x \in \mathscr{X}(\op{Spec}(K)) $ where $K / \mathbb{C}$ is an extension and let $P:U \rightarrow \mathscr{X} $ be a smooth morphism from a scheme.  We assume $u \in U $ maps to $x \in |\mathscr{X}|$. Then
\[
\op{dim}_{x}(\mathscr{X}) := \op{dim}_u(U)-\op{dim}_u(P).
\]	
\end{dfn}

\begin{rmk}
If there is no confusion, we do not distinguish $\mathscr{X}$ with $|\mathscr{X}|$. And, Irreducible decomposition of $\mathscr{X}$ means irreducible decomposition of $|\mathscr{X}|$.
\end{rmk}

\subsection{Harder-Narasimhan filtrations and polygons}
\begin{thm}[Harder-Narasimhan(HN) filtration \cite{huybrechts2010geometry}]
		Let $X$ be a projective surface over $\mathbb{C}$, let $H$ be an ample divisor on $X$ and let $E$ be a torsion free sheaf on $X$. Then, for $E$ and $H$, there exists a unique filtration 
		\[
		0 = E_0 \subset E_1 \subset \cdots \subset E_{s-1} \subset E_s =E 
		\]
	  such that $E_i/E_{i-1}$ is $\mu$-semistable for $H$ $(i=1,\cdots s)$ and 
\[
 \mu(E_1/E_0)  > \mu(E_2/E_1) > \cdots > \mu(E_{s-1}/E_{s-2}) > \mu(E_{s}/E_{s-1}).
\]
 It is called  Harder-Narasimhan(HN) filtration of $E$ for $\mu$-stability.
		
		In the same way, we have Harder-Narasimhan filtration of $E$ for stability.			\end{thm}
	
	\begin{dfn}[Harder-Narasimhan polygon{\cite{nitsure2011schematic}},{\cite{shatz1977decomposition}}]
		 Let $X$ be a projective surface over $\mathbb{C}$, let $H$ be an ample divisor on $X$ and let $E$ be a torsion free sheaf on $X$. We assume that $E$ has the HN filtration for $\mu$-stability	
		 \[
		0 = E_0 \subset E_1 \subset \cdots \subset E_{s-1} \subset E_s =E.
		\]
		Then, we define the Harder-Narasimhan(HN) polygon $\op{HNP}(E)$ of $E$ to be the polygon whose vertexes are $(0, 0)$, $(\op{rk}(E_1), \op{deg}_H(E_1)), (\op{rk}(E_2), \op{deg}_H(E_2))$, $\cdots$, $(\op{rk}(E_{s-1})$, $\op{deg}_H(E_{s-1}))$, $(\op{rk}(E)$, $\op{deg}_H(E))$. 
\end{dfn}
\begin{rmk}
We can also define the HN polygon of $E$ for stability. We use the notions of HN-polygon for both stability and $\mu$-stability. (in detail, see \cite{nitsure2011schematic})
\end{rmk}

\section{Irreducible decomposition of $\tf$}

\begin{nota}
In this and next section, $X$ always means a K3 surface of $\rho(X)=1$ and $H$ means the ample generator of $\op{Pic}(X)$.
   We denote the open substack of semi stable sheaves and of $\mu$-semi stable sheaves of $\tf$ by $\sss$ and $\msss$.
If $\overline{\{p\}} \ni p'$, then we write $p \rightsquigarrow p'$, where $p, p'$ denote points of a topological space and say that $p$ specializes $p'$.
\end{nota}
\begin{dfn} \label{dfn:mshn}  
we define $\mathscr{M}_{(v_1,v_2)}^{\text{HN}}(v)$ to be a substack of $\tf$ whose objects and morphisms are defined as follows.
	
Objects: $E$ $\in \tf$ such that  $E$'s HN-filtration is  $ 0 \subset E_1 \subset E$ with $v(E_1)= (1,d_1H,a_1)$, $v(E/E_1)= (1, d_2H, a_2),$ where $v_i:=(1, d_iH, a_i) \in \mathbb{Z} \oplus \op{Pic}(X) \oplus \mathbb{Z}, (i=1,2)$ ;

Morphisms: $\alpha:E \rightarrow E'$: an isomorphism preserving their HN-filtrations.
\end{dfn}

\begin{nota}
Let $v$ be an element of $ \mathbb{Z} \bigoplus \op{Pic}(Z) \bigoplus \mathbb{Z}$. We define 
\begin{align*}
Quot_X(F, v) &:= \{F \twoheadrightarrow E \mid E:\text{coherent on }X, v(E)=v \} ,\\
 R^{N,m}(v) &:= \left\{\varphi : \mathscr{O}_X(-m)^{\oplus N} \twoheadrightarrow E \in Quot_X(\mathscr{O}_X(-m)^{\oplus N}, v)  \mathrel{} \middle| \mathrel{} \begin{gathered} H^0(\varphi(m)):\text{isomorphism} \\  H^i(E(m)) = 0 (i > 0)
 \end{gathered}
	\right\}, \\
    R_{\text{tf}}^{N,m} 
     &:= R^{N,m} \times_{\mathscr{M}(v)} \tf , \\
  R_{\text{ss}}^{N,m} &:= R^{N,m} \times_{\mathscr{M}(v)} \sss \simeq  R_{\text{tf}}^{N,m} \times_\tf \sss , \\
    R_{(v_1, v_2)}^{N,m} 
    &:= R^{N,m} \times_{\mathscr{M}(v)} \mathscr{M}_{(v_1, v_2)}^{\text{HN}}(v) \simeq R^{N,m} \times_{\tf} \mathscr{M}_{(v_1, v_2)}^{\text{HN}}(v).
     \end{align*}

\end{nota}

\begin{rmk}
 $[ R_{\text{ss}}^{N,m}/\op{GL}(N)] \rightarrow \sss$ and $[R_{(v_1, v_2)}^{N,m}/\op{GL}(N)] \rightarrow \hn$ are open immersions because $[R^{N,m}/\op{GL}(N)] \rightarrow \mathscr{M}(v)$ is an open immersion (\cite[Proposition 9.6]{joyce2012theory}). 
In addition, we have $\op{dim}[ R_{\text{ss}}^{N,m}/\op{GL}(N)] = \op{dim}R_{\text{ss}}^{N,m} - \op{dim}\op{GL}(N)$ and $\op{dim}[R_{(v_1, v_2)}^{N,m}/\op{GL}(N)] = \op{dim}R_{(v_1, v_2)}^{N,m} - \op{dim}\op{GL}(N)$.
\label{rmk:quotient stack}
\end{rmk}

\subsection{Irreducibility of moduli stacks of sheaves and known results}
In this subsection, we refer to irreducibility of moduli stacks of HN-filtrations and known results needed to prove the our results.

\begin{lemm}[{\cite[Theorem 1.2]{kimura2011}}]
\label{lemm:irr,yoshioka}Let $X$ be a K3 surface of Picard number 1. If $\la v,v \ra > 0$, then $\sss$ is an irreducible algebraic stack.
		\qed
\end{lemm}
\begin{rmk}\label{lem:irr 2}
When $\la v,v \ra \leq 0$ and $\sss \neq \emptyset$, the topological spaces of moduli stacks and moduli schemes are  homeomorphic because the stacks are quotient stacks and all semistable sheaves are polystable. Therefore, the moduli stacks are irreducible.
\end{rmk}

\begin{lemm}[{\cite[Lemma 2.5]{yoshioka2009fourier}}]
Let $\mathscr{M}_{(v_1, v_2)}^{\text{HN}}(v)$ be the moduli stack of torsion-free sheaves with Mukai vector $v$ whose Harder-Narasimhan type is $(v_1, v_2)$. Then
\begin{enumerate}
\item the morphism $\mathscr{M}_{(v_1, v_2)}^{\text{HN}}(v) \rightarrow \mathscr{M}(v)$ is an immersion;
\item Let $E \in \mathscr{M}_{(v_1, v_2)}^{\text{HN}}(v)$, whose HN-filtration corresponds to  $0 \rightarrow F_1 \rightarrow E\rightarrow F_2 \rightarrow 0  $ and let $\mathscr{M}_{(v_1, v_2)}^{\text{HN}}(v) \rightarrow \mathscr{M}^{\text{ss}}(v_1) \times \mathscr{M}^{\text{ss}}(v_2)$ be a morphism which sends $[E] \mapsto ([F_1], [F_2])$. Then, all irreducible components of $\mathscr{M}_{(v_1, v_2)}^{\text{HN}}(v)$ are obtained as the pullback of an irreducible component of $\mathscr{M}^{\text{ss}}(v_1) \times \mathscr{M}^{\text{ss}}(v_2)$.
\end{enumerate}
\end{lemm}
\begin{cor}
	\label{cor:irr HN}
	$\mathscr{M}_{(v_1, v_2)}^{\text{HN}}(v)$ is an irreducible algebraic stack.\label{irreducible}
\end{cor}
 We explain facts about $\op{dim}\mathscr{M}^{\text{ss}}(v)$ and $\op{dim}\mathscr{M}_{(v_1, v_2)}^{\text{HN}}(v)$, which are necessary to prove a proposition later.
\begin{lemm}[{\cite[Theorem 1.2]{kimura2011}}
{\cite[Lemma 5.3]{kurihara2008holomorphic}}, 
{\cite[Lemma 5.3.2]{minamide2018wall}}]
	Let $X$ be K3 surface of Picard number 1. $v = lv_0$ with  $v_0:\text{primitive}$ and $l \in \mathbb{Z}$. Then, 
	\begin{align*}
		&\op{dim}\mathscr{M}^{\text{ss}}(v)=\begin{cases} \la v, v \ra +1 & \la v, v \ra  > 0\\
		\la v, v \ra + l & \la v, v \ra =0 \\
		\la v, v \ra + l^2 & \la v_0, v_0 \ra = -2
	\end{cases},
 &&\op{dim}\mathscr{M}_{(v_1, v_2)}^{\text{HN}}(v) =\la v_1, v_1 \ra + \la v_2, v_2 \ra + \la v_1, v_2 \ra +2.
	\end{align*}
	\label{lemm:dimension formula}
	\end{lemm}
\vspace{-3mm}	

We also explain facts which are necessary to prove Theorem \ref{thm:mainthm1}.
	\begin{lemm}[({\cite[Proposition 1.1]{kimura2011}})]
  	The dimensions of all irreducible components of $\mathscr{M}(v)$ is more than(or equal to)  $\la v, v\ra +1 $.
	\label{lemm:dimbound}
	\end{lemm}

	\begin{lemm}[{\cite[Lemma 2.21]{emerton2017dimension}}]
		Let $\mathscr{X}$ be a pseudo-catenary, jacobson, and locally noetherian algebraic stack. If $|\mathscr{X}|$ is irreducible, then $\op{dim}_x\mathscr{X}$ is constant for all $x \in |\mathscr{X}|$.\label{lemm:dimconst}
	\end{lemm}

\begin{rmk}\label{rmk:lowerbound}
\begin{enumerate}[(i)]
	\item algebraic stacks which are locally of finite type satisfy the assumption of Lemma \ref{lemm:dimconst}.
	\item by Lemma \ref{lemm:dimbound} and Lemma \ref{lemm:dimconst}, we get $\op{dim}_x\mathscr{M}^{tf}(v) \geq \la v, v\ra +1 $ $(\forall x \in |\mathscr{M}^{tf}(v)|).$
\end{enumerate}
\end{rmk}	

\subsection{A criterion of the irreducible components of $\tf$}
In this section, we classify the irreducible components of $\tf$. 

\begin{lem}
Let $v_1, v_2, v_1', v_2' \in \mathbb{Z} \bigoplus \op{Pic}(X) \bigoplus \mathbb{Z}$.
We suppose  $v_1 \neq v_1'$ and $v_2 \neq v_2'$.
	For the stacks  $\mathscr{M}_{(v_1, v_2)}^{\text{HN}}(v)$ and $\mathscr{M}_{(v_1', v_2')}^{\text{HN}}(v)$, 
We have $\overline{\mathscr{M}_{(v_1, v_2)}^{\text{HN}}(v)} \nsubseteq \overline{\mathscr{M}_{(v_1', v_2')}^{\text{HN}}(v)}$.\label{lem:lll}
\label{lem:inclusion}
\end{lem}

\begin{proof}
We assume that $\overline{\mathscr{M}_{(v_1, v_2)}^{\text{HN}}(v)} \subseteq \overline{\mathscr{M}_{(v_1', v_2')}^{\text{HN}}(v)}$ holds. If let $p$ and $ p'$ be the generic points of $\overline{\mathscr{M}_{(v_1, v_2)}^{\text{HN}}(v)}$ and $\overline{\mathscr{M}_{(v_1', v_2')}^{\text{HN}}(v)}$ respectively, 
there exist $ N, m \in \mathbb{Z}_{\geq 0}$ such that the morphism $R_{(v_1, v_2)}^{N, m}\rightarrow \overline{\mathscr{M}_{(v_1, v_2)}^{\text{HN}}(v)}$ is dominant, i.e., $R_{(v_1, v_2)}^{N, m} \ni \exists q \mapsto p \in \overline{\mathscr{M}_{(v_1, v_2)}^{\text{HN}}(v)}$.
Note that $R_{(v_1, v_2)}^{N, m}$ is irreducible.

By the fact that $p' \rightsquigarrow p$ and \cite{laumon2000moret} , there exists $q' \in R_{(v_1', v_2')}^{N,m}$ such that $R_{(v_1', v_2')}^{N,m} \ni q' \mapsto p' \in \overline{\mathscr{M}_{(v_1, v_2)}^{\text{HN}}(v)}$ and $q' \rightsquigarrow q$ in $R_{\text{tf}}^{N,m}$ and we think of $q'$ as the generic point of $R_{(v_1', v_2')}^{N,m}$. So, we get $\op{dim}\overline{R_{(v_1, v_2)}^{N,m}} = \op{dim}R_{(v_1, v_2)}^{N,m}$ and $\op{dim}\overline{R_{(v_1', v_2')}^{N,m}} = \op{dim}R_{(v_1', v_2')}^{N,m}$. 
On the other hand, we have $\mathscr{M}_{(v_1, v_2)}^{\text{HN}}(v)$ $\neq$ $\mathscr{M}_{(v_1', v_2')}^{\text{HN}}(v)$ .
Therefore, it holds that $\op{dim}R_{(v_1', v_2')}^{N,m}> \op{dim}R_{(v_1, v_2)}^{N,m}$.
$($ if $\op{dim}R_{(v_1', v_2')}^{N,m} = \op{dim}R_{(v_1, v_2)}^{N,m}$, we have $\overline{R_{(v_1', v_2')}^{N,m}} = \overline{R_{(v_1, v_2)}^{N,m}}$, this contradicts to uniqueness of generic points.$)$





If $v_1 := (1, mH, \frac{m^2H^2}{2}-\ell_1+1) $ and $v_2 := ( 1, (n-m)H, \frac{(n-m)^2H^2}{2}-\ell_2 +1) $, then 
\begin{small}
\begin{align*}
	\op{dim}\mathscr{M}_{(v_1, v_2)}^{\text{HN}}(v) &= \la v_1, v_1 \ra + \la v_2, v_2 \ra + \la v_1, v_2 \ra +2 \\
&= (2\ell_1-2) + (2\ell_2-2) + \left\{ m(n-m)H^2 - \left( \frac{m^2H^2}{2}-\ell_1 +1 \right) - \left(\frac{(n-m)^2H^2}{2}-\ell_2 +1 \right) \right \} +2\\
&= 3(\ell_1+\ell_2)-4 + m(n-m)H^2 - \frac{m^2H^2}{2} -\frac{(n-m)^2H^2}{2}\\
&= -2m(n-m)H^2 - \frac{m^2H^2}{2} -\frac{(n-m)^2H^2}{2}+(3c_2-4) = H^2\left(m-\frac{n}{2}\right)^2+3c_2-4-\frac{3n^2H^2}{4} .
\end{align*}
\end{small}
Note that $\ell_1 + \ell_2 +m(n-m)H^2 = c_2$ in the above calculation.

On the other hand, the map $|\tf| \ni p'' \mapsto$ HNP($p''$) is upper semicontinuous  by \cite{shatz1977decomposition} or \cite{nitsure2011schematic}, where HNP($p''$) := HNP($E''$)($E''$ is the corrsponding object in $\tf$ to $p''$).
So, we have HNP($p$) $\geq$ HNP($p'$) because $p' \rightsquigarrow p$. 
This means $m \geq m'$ ($v'_1 := \la 1, m'H, \frac{m'^2H^2}{2}-\ell'_1+1 \ra$).
And, we have $\op{dim}\mathscr{M}^{\text{HN}}_{(v_1,v_2)}(v) \geq \op{dim}\mathscr{M}^{\text{HN}}_{(v'_1,v'_2)}(v)$ by the above calculation.
This is equivalent to $\op{dim}R_{(v_1, v_2)}^{N,m} \geq \op{dim}R_{(v_1', v_2')}^{N,m}$ because of the irreducibility of $\hn$ and Remark \ref{rmk:quotient stack}.
This is a contradiction. 
Therefore, we get the proposition.
\end{proof}

\begin{rmk}
\label{180857_3Dec21}
\begin{enumerate}
 \item It is shown that $\overline{\sss} \supseteq \overline{\hn} $ implies $\op{dim}\sss > \op{dim}\hn$ by the method of the first half of the proof of Lemma \ref{lem:inclusion}, Remark \ref{rmk:quotient stack} and the irreducibility of $\sss$ and $\hn$.
\item If $\sss \neq \emptyset$, $\overline{\sss} \nsubseteq \overline{\mathscr{M}_{(v_1, v_2)}^{\text{HN}}(v)}$.
If $\overline{\sss} \subseteq \overline{\mathscr{M}_{(v_1, v_2)}^{\text{HN}}(v)}$, then we have HNP($p_1$) $\geq$ $\op{HNP}(p_2$) where $p_1$(resp.$p_2$) is the generic point of $\sss(\op{resp}.\hn)$. 
However this does not occur.
 
\end{enumerate}

\end{rmk}

\subsection{The proof of Theorem \ref{thm:mainthm1} 
}
We prove Theorem \ref{thm:mainthm1} 
by using the lemmas before. 

\begin{proof}
Any Mukai vector $v$ satisfies one of the following disjoint conditions ;
\begin{align*}
&(a) : \langle v,v \rangle > 0, &&(b) : \la v, v \ra=0, -2 \text{ and } v \text{ is primitive }, \\ 
&(c):  \la v_0, v_0 \ra=0, -2 \text{ and } v \text{ is non-primitive }, &&(d) : \langle v,v \rangle < -2 \text{ and } \la v_0, v_0 \ra \neq -2. 
\end{align*}
And, we will prove Theorem \ref{thm:mainthm1} in each case.
	We first have a stratification of $\tf$ by $\sss$ and $\mathscr{M}_{(v_1, v_2)}^{\text{HN}}$ $(v)$. (About HN stratification, for example, see \cite{nitsure2011schematic} or \cite{hoskins2018stratifications}, Section 5). 
	In the case of $(a)$ and $(b)$, if $\la v_1, v_2\ra \leq 1$, 
	\begin{align*}
		\op{dim}\sss &= \la v, v \ra +1
		= \la v_1, v_1 \ra +\la v_2, v_2 \ra + 2\la v_1, v_2 \ra + 1
		= \op{dim}\mathscr{M}_{(v_1, v_2)}^{\text{HN}}(v)+ \la v_1, v_2 \ra - 1.
	\end{align*}
	 By Remark \ref{180857_3Dec21}, we get $\overline{\sss}\nsupseteq \overline{\mathscr{M}_{(v_1, v_2)}^{\text{HN}}(v)}$ and $\overline{\sss}\nsubseteq \overline{\mathscr{M}_{(v_1, v_2)}^{\text{HN}}(v)}$. 
	  On the other hand, we consider the case $\la v_1, v_2\ra > 1$. 
We assume $\overline{\sss}\nsupseteq \overline{\mathscr{M}_{(v_1, v_2)}^{\text{HN}}(v)}$. 
Then, for general $ x \in \overline{\mathscr{M}_{(v_1, v_2)}^{\text{HN}}(v)}$, we have $\op{dim}_x\tf < \la v, v\ra +1$ and this contradicts Remark \ref{rmk:lowerbound}. 
So we have $\overline{\sss}\supseteq \overline{\mathscr{M}_{(v_1, v_2)}^{\text{HN}}(v)}$. By Lemma \ref{lemm:irr,yoshioka}, Remark \ref{lem:irr 2}, 
	 and Corollary \ref{cor:irr HN}, the stacks $\sss$ and $\hn$ are irreducible. Therefore by Lemma \ref{lem:inclusion}, we can classify the irreducible components of $\tf$ as the statements of the theorem.
  In the case of $(c)$, we can show that for any stack of HN-filtration, $\la v_1, v_2 \ra \leq 0$ and $\op{dim}\hn \geq \op{dim}\sss$.

 Let $v := ( 2, nH, \frac{n^2H^2}{2}-c_2+2 )$. 
Then, we have
\[
 \frac{1}{4}\la v, v \ra = \la v_0, v_0 \ra = \frac{-n^2H^2}{4}+c_2-2 = 0 \text{ or } -2. 
\]
And Let $v_1 := (1, kH , \frac{k^2H^2}{2}-\ell_1+1), v_2 := (1, lH , \frac{l^2H^2}{2}-\ell_2+1)$.
Then, 
\small{
\begin{align*}
 \la v_1, v_2 \ra &= klH^2 -\frac{n^2H^2}{2}+c_2-2 
  = klH^2-\frac{n^2H^2}{4}+(\frac{n^2H^2}{4}+c_2-2) = -\frac{H^2}{4}(n^2-4kl)+ (\frac{n^2H^2}{4}+c_2-2) \\
&= -\frac{H^2}{4}(k-l)^2 + (\frac{n^2H^2}{4}+c_2-2) 
= -\frac{H^2}{4}(k-l)^2 + \begin{cases}
			   0 & \la v_0, v_0 \ra = 0 \\
                          -2 & \la v_0, v_0 \ra =-2
			  \end{cases}  \leq 0 .
\end{align*}
}
So, we have
$\op{dim}\hn = \la v, v \ra - \la v_1, v_2 \ra + 2 \geq \op{dim}\sss$.
we get $\overline{\sss}\nsupseteq \overline{\mathscr{M}_{(v_1, v_2)}^{\text{HN}}(v)}$ and $\overline{\sss}\nsubseteq \overline{\mathscr{M}_{(v_1, v_2)}^{\text{HN}}(v)}$ for any $(v_1, v_2)$ by Remark \ref{180857_3Dec21}.

	 In the case (d), we have $\sss= \emptyset$ by \cite[Cor 0.3]{yoshioka1999irreducibility}. So, we can classify the irreducible components.
	 	 \end{proof}
	
	\section{An application to Brill-Noether theory of Hilbert schemes of points}
	In \cite{walter1995components}, an application of the irreducible components of moduli stacks of torsion free sheaves on ruled surfaces are performed. In this section, we replace ruled surfaces by K3 surfaces. For a K3 surface $X$, let $N$ be a non-negative integer and let $D$ be an effective divisor on $X$ such that $h^0(X, \mathscr{O}(D)) \geq N$. And let $\op{Hilb}^N(X)$ be the Hilbert scheme of finite schemes of length $N$ on $X$. For the Hilbert schemes $\op{Hilb}^N(X)$ of finite schemes of length $N$ on $X$, We define $W_N^i(D)$ as follows.
\[
W_N^i(D) := \{ [Z] \in \op{Hilb}^N(X) \mid h^1(\mathscr{I}_Z(D)) \geq i+1\}.
\]
Then, it is known that $ W_N^i(D) \subseteq \op{Hilb}^N(X)$ is a closed subscheme from upper semicontinuity of cohomology of flat families of sheaves and $h^1(\mathscr{I}_Z(D)) = i+1 $ for general members of each irreducible component of $W_N^i(D)$. In particular, if $i=0$, we have a bijection between the irreducible components of $ W_N^i(D)$ and the irreducible components of $\mathscr{M}^{\text{tf}}(v)$ whose general member $E$ satisfies the conditions  $(1): H^1(X, E)=H^2(X, E)=0$ and $(2):\exists s \in H^0(X, E)$ such that $E / s \mathscr{O}_X$ is torsion free. where, $v := (2, D, \frac{D^2}{2}-N+2)$. Note that the conditions (1) and (2) are open conditions. Moreover, if $E$ is a general member of an irreducible component $\mathscr{M}'$ of $\mathscr{M}^{\text{tf}}(v)$ which satisfies $(1), (2)$ and let the corresponding irreducible component of $W_N^i(D)$ be $V$, then 
 \begin{equation}
 \op{dim}V = \op{dim}\mathscr{M}'+h^0(E).
 \tag{$\spadesuit$} \label{dimfom} 
 \end{equation}
\begin{proof}[Proof of Theorem \ref{thm:mainthm2}]
	We get the claim of Theorem \ref{thm:mainthm2} by the above comment, Lemma \ref{lem:notsemistable}, Lemma \ref{lem:semistable} and calculating and rearranging  $\chi(v) > 0$ and $h^0(\mathscr{O}_X(n-m)) > \ell_2$. 

For example, a not semistable component $\overline{\hn} \subset \tf$ correponds to a component of $W^0_N(nH)$ if and only of the folloing conditions hold : 
\begin{itemize}
\item  $\la v_1, v_2 \ra \leq -1 (\text{ because }\hn$ is an irreducible component of $\tf$),
\item $2m \geq n > m > 0 \Leftrightarrow [v_1]_1 > 0, [v_2]_2 > 0$,
\item $\chi(v) > 0 \text{ (This always holds by the assumption } N \leq h^0(\mathscr{O}(n)) \text{ and the Riemann-Roch formula)} $,
\item $h^0(\mathscr{O}_X(n-m)) > \ell_2 \Leftrightarrow -1 < [v_2]_2$,
\end{itemize}
 where $v_1:=(1, mH, \frac{m^2H^2}{2}-\ell_1+1)$, and let $v_2:= (1, (n-m)H, \frac{(n-m)^2H^2}{2}-\ell_2+1)$.
Thus, we have ($\alpha$) of Theorem \ref{thm:mainthm2}.
In the same way, we have ($\beta$) of Theorem \ref{thm:mainthm2}.
\end{proof}

\subsection{About not semistable components}
\begin{lem}Let $v:=(2, nH, \frac{n^2H^2}{2}-N+2)$, let 
$v_1:=(1, mH, \frac{m^2H^2}{2}-\ell_1+1)$, and let $v_2:= (1, (n-m)H, \frac{(n-m)^2H^2}{2}-\ell_2+1)$ such that $v=v_1+v_2$.
	We assume that $E$ is a general member of $\mathscr{M}_{(v_1, v_2)}^{\text{HN}}(v)$. Then, $E$ satisfies the conditions (1), (2) if and only if the following conditions hold. $(a):2m \geq n > m > 0$, $(b):\chi(E) > 0$, $(c):h^0(\mathscr{O}_X(n-m)) > \ell_2$. \label{lem:notsemistable}
\end{lem}

\begin{proof}
If the conditions (1), (2) are satisfied, it is clear that a general $E$ satisfies (b). 
If let the HN-filtration of $E$ be the sequence
\[
 0 \rightarrow \mathscr{I}_{Z_1}(m) \rightarrow E \rightarrow \mathscr{I}_{Z_2}(n-m) \rightarrow 0
\]
, where $v(\mathscr{I}_{Z_1}(m)) = v_1, v(\mathscr{I}_{Z_2}(n-m)) = v_2$.
We have (a) because $E \in \hn$ and $h^2(\mathscr{O}(m)) = h^2(\mathscr{O}(n-m)) =  0$ holds.
We also have $h^0(\mathscr{O}(n-m)) > \ell_2$ because $h^1(\mathscr{I}_{Z_2}(n-m)) = 0$ and we have $0 \rightarrow \mathscr{I}_{Z_2}(n-m) \rightarrow \mathscr{O}(n-m) \rightarrow \mathscr{O}_{Z_2} \rightarrow 0$.
Note that the condition $h^0(\mathscr{O}(n-m)) = \ell_2$ can not occur. 
If $h^0(\mathscr{O}(n-m)) = \ell_2$, then any global section $s$ of general sheaf $E \in \hn$ is included in $H^0(\mathscr{I}_{Z_1}(m))$.
Because all non-zero sections in $H^0(\mathscr{I}_{Z_1}(m))$ never induce torsion free quotients, so $E/s\mathscr{O}_X$ include a torsion sheaf $\mathscr{I}_{Z_1}(m)/s \mathscr{O}_X$. 
This contradicts to the condition (2).
So, we have the condition (c).

Conversely, we assume that the conditions (a), (b) and (c) are satisfied.
First, we prove the condition (1) by induction for $\ell_1$ (cf. \cite[Lem 3.3 and Lem 4.5]{walter1995components}). Note that $H^2(E) = 0$ for a general $E \in \hn$ because $H^2(\mathscr{I}_{Z_1}(m)) = H^2(\mathscr{I}_{Z_2}(n-m)) = 0$. If $\ell_1 = 0$, then we have $H^1(E) = 0$ in the same way. For general $\ell_1 > 0$,  we prove $H^1(E) = 0$ for a general $E \in \hn$. We assume $E'$ fits in the exact sequence
\[
0 \rightarrow \mathscr{I}_{Z'_1}(m) \rightarrow E' \rightarrow \mathscr{I}_{Z'_2}(n-m) \rightarrow 0
\]
which is the HN filtration of $E'$ with $\ell(Z'_1) = \ell_1-1$ and $\ell(Z'_2) = \ell_2$. If $H^1(E') = 0$ and $E'$ satisfies the conditions (a), (b) and (c), then, $E'$ have a nonzero global section $s$. And, for a general point $x \in X$ and a general one dimensional quotient $E' \twoheadrightarrow E' \otimes k(x) \twoheadrightarrow k(x)$ of the fiber of $E'$ at $x$ denoted by $\varphi$, we have $\varphi(s) \neq 0$. Note that we can assume $x \notin Z'_1$.  Let $E$ be the kernel of $\varphi$. Then, we have $h^0(E) = h^0(E')-1$ and $H^1(E) = 0$. And, we get the HN filtration of $E$
\[
0 \rightarrow \mathscr{I}_{Z'_1 \cup \{x\}}(m) \rightarrow E \rightarrow \mathscr{I}_{Z'_2}(n-m) \rightarrow 0
\]
because of the HN-filtration of $E'$ and the assumption $x \notin Z'_1$. So, we get condition (1) for general $\ell_1 > 0$.

Next, we prove the condition (2) under the condition (1).

We consider the conditions $(\alpha): 2m= n, (\beta): \ell_2 =1$.
And, we divide our proof into two cases:
\begin{itemize}
 \item[(i)] $(\alpha)$ or $(\beta)$ is not true.
\item[(ii)] Both $(\alpha)$ and $(\beta)$ are true.
\end{itemize}

\underline{Case (i)}
It is enough to prove the following claim.

\begin{cl}\label{lem:claim}
	Let $k$ be a positive integer. We assume that $\ell_1 = 0$. Then, we have
	\[
	h^0(E(-k)) + \op{dim}|kH| < h^0(E).
	\]
\end{cl} 

We show Claim \ref{lem:claim} induces the condition (2) before proving it.
If the claim is true, then a general $E \in \hn$ with $\ell_1 = 0$ is a vector bundle because the Cayley-Bacharach property (cf. \cite[Thm 5.1.1]{huybrechts2010geometry}) holds for a pair $(Z_2,\mathscr{O}_X(n-2m))$ by the choice of $m$ and $\ell_2$, where $Z_2$ is a general set of $\ell_2$ points. 
And, the set $H^0(E) \setminus  \bigcup_{
	C \in | kH |,  k \in \mathbb{N}}H^0(E(-C))$  is a non-empty open set from Claim \ref{lem:claim}.
 So, a general section $s$ of a general  $E \in \hn$ with $\ell_1 = 0$ defines a torsion free quotient $E/s\mathscr{O}_X$ because the zero set $Z(s)$ of $s$ is a finite set (cf. \cite[Ch. 1, $\S5$]{okonek2011vector}). 
 
 In the case $\ell_1 >0$, we have a vector bundle $E'$ fitting into the sequence
\[	0 \rightarrow \mathscr{O}_X(m) \rightarrow E' \rightarrow \mathscr{I}_{Z_2}(n-m) \rightarrow 0 \quad (\ell(Z_2) = \ell_2)
\]
 whose general section $s$ determines a torsion free quotient because of the case $l_1 = 0$.  

 In addition, $E'$ is generically globally generated. Note that we say that $E'$ is generically generated if the evaluation map $\op{ev}:H^0(E') \otimes \mathscr{O}_X \rightarrow E'$ is surjective on an open set of $X$. 
 Actually, from the condition (a) and (c), $\mathscr{O}_X(m)$ and $ \mathscr{I}_{Z_2}(n-m)$ is generically globally generated. 
 So, a simple diagram chase shows that $E'$ is generically global generated. 

Let $U$ be the subset of $H^0(E')$ of the sections defining torsion free quotients. 
Then, a natural $\mathbb{C}$-linear homomorphism $\overset{\sim}{\op{ev}}: H^0(E') \rightarrow H^0(E' \otimes k(x))$ obtained from $\op{ev}$ above is surjective and $\overset{\sim}{\op{ev}}|_U$ is dominant for general $x \in X$ because $E'$ is generically globally generated.
So, we can take general $\ell_1$ points $x_1, \cdots , x_{\ell_1}$ on $X$ and a general section $s$ such that
 $s \notin \mathscr{O}_X(m) \otimes k(x_i)$ for all $i$ and $s$ defines a torsion free quotient. 
Then, we can take one-dimensional quotients $\varphi_i:E' \twoheadrightarrow E' \otimes k(x_i) \twoheadrightarrow k(x_i) (1 \leq i \leq \ell_1)$ such that $\varphi_i |_{\mathscr{O}_X(m)} \neq 0$ for all $i$ and $\varphi_i(s) = 0$ for all $i$. 
We consider the quotient $ \varphi: E' \twoheadrightarrow \bigoplus_{i=1}^{\ell_1} k(x_i)$ obtained from $\varphi_i$. 
If let $E$ be the kernel of $\varphi$, then $E \in \hn$ and a general section $s$ of $E$ defines a torsion free quotient.
   
\begin{proof}[Proof of Claim \ref{lem:claim}]
In this paper, we only consider the case 
\begin{align*}
	&2m > n > m > 0, &&
\ell(Z_2) > h^0(\mathscr{O}(n-m-1) = \frac{(n-m-1)^2}{2}H^2+2, \\
 &m \geq 3, &&  n-m \geq 3.
\end{align*}
The other cases can be proved in the same way or more easily.

Because $h^0(\mathscr{O}(n-m-1) < \ell(Z_2) < h^0(\mathscr{O}(n-m))$,  we have $H^0(\mathscr{I}_{Z_2}(n-m-k))=0$ for all positive integer $k$ and general $Z_2$. So, we have $H^0(E(-k)) = H^0(\mathscr{O}(m-k))$. 

In this condition, we have $H^0(E(-k)) = \chi(\mathscr{O}(m-k)) = \frac{(m-k)^2}{2}H^2 + 2 $ and $\op{dim}|kH| = h^0(\mathscr{O}(kH)) -1 = \frac{k^2}{2}H^2 + 1$.
Note that $h^0(E) = \chi(E) 
            = \frac{m^2}{2}H^2+\frac{(n-m)^2}{2}H^2+4-\ell(Z_2)$. Then, we can calculate as follows.
\begin{small}
\begin{align*}
h^0(E) - \{ h^0(E(-k)) + \op{dim}|kH| \} &= \frac{m^2}{2}H^2+\frac{(n-m)^2}{2}H^2-\frac{(m-k)^2}{2}H^2 -\frac{k^2}{2}H^2
+1-\ell(Z_2).
\end{align*}
In addition, we have $\ell(Z_2) < h^0(\mathscr{O}(n-m)) = \frac{(n-m)^2}{2}H^2+2$. So,  
\begin{align*}
	h^0(E) - \{ h^0(E(-k)) + \op{dim}|kH| \}
	 &> H^2k(m-k)-1 > 0   (\because k, m-k >0). 
\end{align*}
\end{small}
\end{proof}
Thus, we get the condition (2) when $2m \neq n$ or $\ell_2 \neq 1$.

\underline{Case (ii)}
Next, we suppose $(\alpha)$ and $(\beta)$ are true.
 In this case, note that $\ell_1 = 0$ because $\chi(v_1) > \chi(v_2)$ and every sheaf $E \in \hn$ is isomorphic to $\mathscr{O}_X(m) \oplus \mathscr{I}_x (m)$ for some $x \in X$. 
We take section $s_1, s_2 \in H^0(\mathscr{O}(m))$ such that $Z(s_1) \cap Z(s_2)$ is a finite set, where $Z(s_i)$ is the zero set of $s_i (i = 1,2)$. If $x \in Z(s_2)$, $s_1 \oplus s_2 \in H^0(\mathscr{O}_X(m) \oplus \mathscr{I}_x (m))$. 
Since $s_1 \oplus s_2 \in \mathscr{O}_X(m) \oplus \mathscr{O}_X(m)$ defines a torsion free quotient $\mathscr{O}_X(m) \oplus \mathscr{O}_X(m) / (s_1 \oplus s_2) \mathscr{O}_X$, $\mathscr{O}_X(m) \oplus \mathscr{I}_x(m) / (s_1 \oplus s_2) \mathscr{O}_X$ is also torsion free.
\end{proof}

\subsection{About semistable components}
We will use the following lemmas to prove the Lemma \ref{lem:semistable}.

 \begin{lemm}[\cite{yoshioka1999irreducibility} Lemma 1.4 or \cite{yoshioka1999some} Lemma 2.1]
  Let $n$ be an odd integer. If the exact sequence
  \begin{small}
 	\[ 0 \rightarrow \mathscr{O} \left( \frac{n-1}{2} \right) \rightarrow E \rightarrow \mathscr{I}_Z\left( \frac{n+1}{2} \right)\rightarrow 0
 	\]
  \end{small}
 	does not split, then $E$ is a $\mu$-stable sheaf, where $\mathscr{I}_Z$ is the ideal sheaf of a finite subscheme $Z$. 
 	\label{lem:exactseq}
 	\end{lemm}
 	
 \begin{lemm}[\cite{yoshioka1999irreducibility} Proposition 0.5 and Section 3.3]Let $v:=(2, nH, \frac{n^2H^2}{2}-N+2)$. We assume that $v$ is primitive, $v \neq (2, 0, -1)$ and “ $v \neq (2,nH, \frac{n^2H^2}{4}-1)$ and $n$ is even”.  
 Then, there exists a stable locally free sheaf with Mukai vector $v$.
 	\label{lem:vctbdl}
 \end{lemm}

\begin{rmk}
\label{172605_3Dec21}
 In the Lemma \ref{lem:vctbdl}, if $n$ is odd, then any stable sheaf is $\mu$-stable sheaf. 
However, if $n$:even, a stable sheaf is not necessarily a $\mu$-stable sheaf.
\end{rmk}

\begin{lem}
	Let $E$ be a general member of the stack $\mathscr{M}^{\text{ss}}(v)$. Then, the conditions (1) and (2) are equal to the conditions $\chi(E) > 0$ and  “$H^2 \neq 2$ or $ v \neq (2,3H, 5)$”.
	\label{lem:semistable}
\end{lem}

 \begin{proof}
 If (1) and (2) satisfy, we have $H^1(E)=0$ and $H^0(E) \neq 0$. Therefore, we have $\chi(E)>0$. 

We will prove Lemma \ref{lem:semistable} only when $n$ is an odd integer. 
We can also prove this lemma in the same way when $n$ is even. Note, for a general $E$, we have $H^2(E) = 0$ by semistability. 

In the following we assume only $\chi(v) > 0$ (we do not assume the latter condition).

\subsection*{When $N > \frac{n^2+1}{4}H^2 +3 $ with odd $n$
}
\subsubsection*{About the conditions (1) and (2)}
First, we assume that $N > \frac{n^2+1}{4}H^2 +3$. This is equivalent to the condition that the closure of the stacks of Harder-Narasimhan filtrations whose general sheaf is an extension
\begin{small}
\[
0 \rightarrow \mathscr{I}_{Z_1}\left( \frac{n+1}{2} \right)  \rightarrow E \rightarrow \mathscr{I}_{Z_2} \left(\frac{n-1}{2} \right) \rightarrow 0
\]
\end{small}
is contained in the closure of $\sss$. Then, we can show that some $E$ in the closure of $\sss$ have no higher cohomology in the same way as in Lemma \ref{lem:notsemistable}. Moreover, we can prove a general $E$ have a global section which give a torsion free quotient.

\subsection*{When  $\frac{n^2+1}{4}H^2 +3 \geq N$ with odd  $n$}

\subsubsection*{About the condition (1)}
Next we assume that $\frac{n^2+1}{4}H^2 +3 \geq N$. From Lemma \ref{lem:vctbdl} and \ref{172605_3Dec21}, there exists a $\mu$-stable sheaf $E$. We next consider $E(-\frac{n-1}{2})$. Let $E':= E(-\frac{n-1}{2})$ and $v' := v(E')$. Then, $E'$ fits into the following exact sequence
\begin{small}
 \[ 0 \rightarrow \mathscr{O} \rightarrow E' \rightarrow \mathscr{I}_Z(1) \rightarrow 0\]
\end{small}
 , where $Z$ is a finite subscheme of $X$.  
Indeed, we have $\op{hom}({E'}^\vee, \mathscr{O}) = \op{ext}^2(\mathscr{O}, {E'}^\vee) = h^2({E'}^\vee) = h^0(E') \neq 0$ and $\op{hom}(E', \mathscr{O}) = \op{ext}^2(\mathscr{O}, E') = h^2(E') = 0$ because $\chi(E') > 0 \text{ and }, E'$ is also $\mu$-stable. So, we have the above exact sequence by using these.
 
Because any non-split extension of $\mathscr{I}_Z(1)$ by $\mathscr{O}$ is a $\mu$-stable sheaf from Lemma \ref{lem:exactseq}, the unique irreducible component of $W^0_{\ell(Z)}(H)$ corresponds to $\mathscr{M}^{\text{ss}}(v')$.
Thus, $h^1(\mathscr{I}_Z(1)) = 1 $, for a general $E$ since $\mathscr{M}^{\text{ss}}(v) \simeq \mathscr{M}^{\text{ss}}(v')$.

Next, we see the following claim holds.
\begin{cl}
  $h^1(\mathscr{I}_{Z'}(2)) = 0$ for general $Z' \in W^0_{\ell(Z)}(H)$ except for the case ``$H^2=\ell(Z)=2$ '' .
 In the case `` $H^2=\ell(Z)=2$ '', $W^0_{\ell(Z)}(H)= W^0_{\ell(Z)}(2H)$ and $h^1(\mathscr{I}_{Z'}(3)) = 0$ for general $Z' \in W^0_{\ell(Z)}(H)$.
\label{163847_3Dec21}
\end{cl}

If  the claim holds, this induces condition (2) except for the case ``$H^2 = 2, v= \la 2 , 3H , 5 \ra$ '' and condition (2) never hold in this exceptional case.
Before proving the claim, we show this.

First, note that ``$ n =3, H^2 = \ell(Z) = 2$'' $ \Rightarrow v= \la 2, 3H, 5 \ra$. So, if the claim holds, any $E \in \sss$ fits into the following exact sequence
\[
 0 \rightarrow \mathscr{O}_X(1) \rightarrow E \rightarrow \mathscr{I}_{Z'}(2) \rightarrow 0
\]
, where $Z' \in W^0_2(H)$. 
Here, $W^0_2(H) = W^0_2(2H)$ from the claim and we have $h^1(\mathscr{I}_{Z'}(2)) \neq 0$ for any $Z' \in W^0_2(H)$. 
Thus, $h^1 (E) \neq 0$ from the long exact sequence of cohomology obtained from above.
This shows the condition (2) never hold in the case `` $H^2 =2, V = \la 2, 3H 5 \ra$.
On the other hand, except for the case, any $E \in \sss$ fit into the follwing exact sequence
\[
 0 \rightarrow \mathscr{O}_X\left( \frac{n-1}{2} \right) \rightarrow E \rightarrow \mathscr{I}_{Z'} \left( \frac{n+1}{2} \right) \rightarrow 0 
\] 
, where $Z' \in W^0_{\ell(Z)}(H)$.
Here, $h^1(\mathscr{I}_{Z'}(2)) = 0$ for general $Z' \in W^0_{\ell(Z)}(H)$ and $h^1(\mathscr{I}_{Z'}(k)) \geq h^1(\mathscr{I}_{Z'}(k+1))$ for all $k$.
This induces the condition (2) hold in general $E \in \sss$.


\begin{proof}[Proof of Claim \ref{163847_3Dec21}]

First, note that we have $\frac{H^2}{2}+3 \geq \ell(Z)$ because $\chi(E') > 0$. 
If $\frac{H^2}{2}+3 = \ell(Z)$, then $\text{Hilb}^{\ell(Z)}(X) = W^0_{\ell(Z)}(H)$. 
For general $Z' \in \text{Hilb}^{\ell(Z)}(X)$, $h^1(\mathscr{I}_{Z'}(2)) = 0$ because $h^0(\mathscr{O}_X(2)) \geq \ell (Z) $. If $\frac{H^2}{2}+2 \geq \ell(Z)$, then $\text{Hilb}^{\ell(Z)}(X) \neq W^0_{\ell(Z)}(H)$. 

Let $v'' := (2, 2H, 2H^2-\ell(Z)+2)$. 
We divide the rest of the proof into 4 steps.
\begin{enumerate}
\item
We have $\mathscr{M}^{\text{ss}}(v'') = \emptyset$ unless “$ H^2 = \ell(Z) =4$”  or “$H^2 = \ell(Z) = 2$” by \cite[Cor 0.3]{yoshioka1999irreducibility}. 
Moreover, there is not an irreducible component of $\mathscr{M}^{\text{tf}}(v'')$ whose general member is a HN-filtration satisfying the conditions (a), (b) and (c) of Lemma \ref{lem:notsemistable} unless “$H^2 = 2$ and $\ell(Z) = 3$”. 
So, from Lemma \ref{lem:notsemistable},  $W^0_{\ell(Z)}(2H) = \emptyset$ except the three cases.  
Thus, $h^1(\mathscr{I}_{Z'}(2)) = 0$ for general $Z' \in W^0_{\ell(Z)}(H)$ except for the three cases.

\item
When ``$H^2 = \ell(Z) = 4$'', $W^0_{\ell(Z)}(2H)$ may not be empty. 
If $W^0_{\ell(Z)}(2H)$ is not empty, the unique irreducible component corresponds to $\mathscr{M}^{\text{ss}}(v'')$. 
Moreover, we can calculate the dimensions of $W^0_{\ell(Z)}(H)$ and $W^0_{\ell(Z)}(2H)$ by using \ref{lemm:dimension formula} and the formula \ref{dimfom} and get $\op{dim} W^0_{\ell(Z)}(H) = 7$ and $\op{dim} W^0_{\ell(Z)}(2H) = 4$.
This means that $W^0_{\ell(Z)}(H) \supsetneqq W^0_{\ell(Z)}(2H)$ and we have the claim.

\item
  When ``$H^2 = \ell(Z) = 2$'', $W^0_{\ell(Z)}(2H)\neq \emptyset$ because the unique point of $\mathscr{M}^{\text{ss}}(v'')$ is $\mathscr{O}_X(H)^{\oplus 2}$. 
And, we also have $\op{dim} W^0_{\ell(Z)}(H) =\op{dim} W^0_{\ell(Z)}(2H) = 2$. 
So, we have $W^0_{\ell(Z)}(H) = W^0_{\ell(Z)}(2H)$. 
However, $W^0_{\ell(Z)}(3H) = \emptyset$ as above. 

\item
In the same way as in Step 2, we get the claim when ``$H^2 = 2$ and $\ell(Z) = 3$''. 

\end{enumerate}

Therefore, we get $h^1(\mathscr{I}_{Z'}(2)) = 0$ for a general $Z' \in W^0_{\ell(Z)}(H)$ when $H^2 \neq 2$ or $\ell(Z) \neq 2$ and $h^1(\mathscr{I}_{Z'}(3)) = 0$ for a general $Z' \in W^0_{\ell(Z)}(H)$ when $H^2 = \ell(Z) = 2$.

\end{proof}

\begin{rmk}
We will explain how we calculate $\op{dim} W^0_{\ell(Z)}(H)$ when `` $H^2 = \ell(Z) =4$'' here (similarly, we can also do when “$H^2 = \ell(Z) = 2$” and “$H^2 = 2, \ell(Z) = 3$”).
 It is sufficient to calculate $\op{dim}\mathscr{M}^{\text{ss}}(v')$, $\op{dim}\mathscr{M}^{\text{ss}}(v'')$, $h^0(E')$ and $h^0(E'')$, where $E''$ is  a general member of $\mathscr{M}^{\text{ss}}(v'')$ from the formula \ref{dimfom}. We can calculate $\op{dim}\mathscr{M}^{\text{ss}}(v')$ and  $\op{dim}\mathscr{M}^{\text{ss}}(v'')$ by using \ref{lemm:dimension formula}, $h^0(E')$ by using the above exact sequence and $h^0(E'')$ by the fact that the unique member of $\mathscr{M}^{\text{ss}}(v'')$ is $\mathscr{O}_X(H)^{\oplus 2}$ (in detail, see \cite{mukai1984moduli}, \cite{kaledin2006singular} et al.).	
\end{rmk}

\subsubsection*{About the condition (2)}
Next,we prove the condition (2). It is enough to prove $h^0(E(-k)) + \op{dim}|kH| < h^0(E)$ as in the same way of the proof of Lemma \ref{lem:notsemistable}   because a general sheaf in $\sss$ is a vector bundle by Lemma \ref{lem:vctbdl}. Note that we have the following exact sequence for a general $E$, 
\begin{small}
	\[ 0 \rightarrow \mathscr{O}\left(\frac{n-1}{2}\right) \rightarrow E \rightarrow \mathscr{I}_Z\left(\frac{n+1}{2}\right) \rightarrow 0\]
\end{small}
, where $Z$ is a finite subscheme of $X$ and $h^1(\mathscr{I}_Z\left(\frac{n+1}{2}\right)) = 0$. So, for $\frac{n-1}{2} \geq k > 0$,
\begin{small}
\begin{align*}
&h^0(E) - \{ h^0(E(-k)) + \op{dim}|kH| \}\\
&\geq (h^0(\mathscr{I}_Z(\frac{n+1}{2})) + \chi(\mathscr{O}(\frac{n-1}{2}))- (h^0(\mathscr{I}_Z(\frac{n+1}{2}-k)) + \chi(\mathscr{O}(\frac{n-1}{2}-k)) + \op{dim}|kH|) \\
&=kH^2(\frac{n-1}{2}-k)-1+h^0(\mathscr{I}_Z(\frac{n+1}{2}))-h^0(\mathscr{I}_Z(\frac{n+1}{2}-k)) > 0 .
\end{align*}
(In the case of $k = \frac{n-1}{2}$, we use $h^0(\mathscr{I}_Z(\frac{n+1}{2})) = \frac{(n+1)^2}{8}H^2 -\ell(Z) +2$ and $ h^0(\mathscr{I}_Z(1)) = \frac{1}{2}H^2 -\ell(Z) +3$ ) 
\end{small}
\begin{rmk}(In the case $n$ is even)
	When $n$ is even, we can prove that a general sheaf $E \in \sss$ have a section defining a torsion free quotient as in the same way as in the proof above except $v = (2, nH, \frac{n^2H^2}{2})$ or $(2,nH, \frac{n^2H^2}{4}-1)$.
	 In these case, any sheaf of $\sss$ is not vector bundle and the closure of $\sss$ dose not contain any stacks of HN-filtration. However, we can prove the condition (1), (2) in the same way of the proof of Lemma \ref{lem:notsemistable}. 
	 In the former case, note that any semistable sheaf is isomorphic to a sheaf of the form $\mathscr{I}_x\left(\frac{n}{2}\right) \oplus \mathscr{I}_y\left(\frac{n}{2}\right)(x, y \in X)$.
	  In the latter case, note that a general quotient $\mathscr{O}\left(\frac{n}{2}\right) \rightarrow \oplus_{i=1}^3 k(x_i) (x_i \in X) $ and any non split extension $0 \rightarrow \mathscr{I}_{\{y_1, y_2\}}\left(\frac{n}{2}\right) \rightarrow E \rightarrow \mathscr{I}_{y_3 }\left(\frac{n}{2}\right) \rightarrow 0(y_j \in X)$ is a semistable sheaf with the Mukai vector $v = (2,nH, \frac{n^2H^2}{4}-1)$ when $n$ is even (cf. \cite[Prop 3.4]{yoshioka1999irreducibility}).
\end{rmk}
 \end{proof}
\begin{small}
\subsection*{Acknowledgements}
	The author would like to thank his advisor Professor Hajime Kaji for helpful comments, warm encouragement and support. He is also grateful to Professors Yasunari Nagai and Ryo Ohkawa for helpful comments and discussions. Finally, he thank the members of the algebraic geometry laboratory of Waseda University. 
\end{small}

\bibliographystyle{amsalpha}
\bibliography{mizuno_paper}

\end{document}